\documentclass{amsart}
\usepackage{amssymb}
\usepackage{mathrsfs}
\usepackage{hyperref}
\usepackage{color}

\date{\today}

\newcommand{\Z}{{\mathbb Z}}
\newcommand{\R}{{\mathbb R}}

\newcommand{\p}{\parallel}

\newcommand{\be}{\begin{equation}}
\newcommand{\ee}{\end{equation}}

\newcommand{\ol}{\overline}
\newcommand{\ti}{\tilde}

\newcommand{\spr}[2]{\langle #1 , #2 \rangle}

\newcommand{\E}{\mathrm{e}}

\newcommand{\tr}{\mathrm{tr}}

\DeclareMathOperator{\dist}{dist}

\newcommand{\eps}{\varepsilon}

\newtheorem{theorem}{Theorem} [section]
\newtheorem{remark}[theorem]{Remark}
\newtheorem{lemma}[theorem]{Lemma}
\newtheorem{proposition}[theorem]{Proposition}
\newtheorem{corollary}[theorem]{Corollary}
\newtheorem{definition}[theorem]{Definition}

\numberwithin{equation}{section}

\begin{document}

\title{Discontinuity of the Lyapunov exponent}

\author[Z.\ Gan]{Zheng Gan}

\address{Department of Mathematics, Rice University, Houston, TX~77005, USA}
\email{\href{mailto:zheng.gan@rice.edu}{zheng.gan@rice.edu}}
\urladdr{\href{http://math.rice.edu/~zg2/}{http://math.rice.edu/$\sim$zg2/}}

\author[H.\ Kr\"uger]{Helge Kr\"uger}

\address{Department of Mathematics, Rice University, Houston, TX~77005, USA}
\email{\href{mailto:helge.krueger@rice.edu}{helge.krueger@rice.edu}}
\urladdr{\href{http://math.rice.edu/~hk7/}{http://math.rice.edu/$\sim$hk7/}}

\thanks{Z.\ G.\ was supported by NSF grant DMS--0800100.}
\thanks{H.\ K.\ was supported by NSF grant DMS--0800100 and a Nettie S. Autrey Fellowship.}

\date{\today}


\begin{abstract}
We study discontinuity of the Lyapunov exponent. We construct a
limit-periodic Schr\"odinger operator, of which the Lyapunov
exponent has a positive measure set of discontinuities. We also
show that the limit-periodic potentials, whose Lyapunov exponent is
discontinuous, are dense in the space of limit-periodic potentials.
\end{abstract}

\maketitle

\section{Introduction}

In this paper, we construct examples of ergodic Schr\"odinger operators
$H_{\omega}$, whose Lyapunov exponent $L(E)$ is discontinuous.
In addition to being a result of interest of its own, we were motivated
by the following observation. Denote by
\be
 \mathcal{Z} = \{E: \quad L(E) = 0\}
\ee
the set where the Lyapunov exponent vanishes. It was shown by
Deift and Simon in \cite{deisi} that the measure of this set
satisfies
\be\label{eq:deisi}
 |\mathcal{Z}| \leq 4.
\ee
Introduce the essential closure of $\mathcal{Z}$ by
\be
 \ol{\mathcal{Z}}^{\mathrm{ess}} = \{E:\quad \forall \eps > 0:\ |\mathcal{Z} \cap (E - \eps, E + \eps)| > 0\}.
\ee
Denote by $\sigma_{\mathrm{ac}}(H_{\omega})$ the absolutely
continuous spectrum of $H_{\omega}: \ell^2(\Z) \to \ell^2(\Z)$.
By Kotani theory, we have that for almost every $\omega$
\be\label{eq:kotani}
 \sigma_{\mathrm{ac}}(H_{\omega}) = \ol{\mathcal{Z}}^{\mathrm{ess}}.
\ee
This result can for example be found in \cite{da1}, \cite{sikot}, or
Theorem~5.17 in \cite{tjac}. Being naive, one might assume that
from these two statements that
\be\label{eq:bigquestion}
 |\sigma_{\mathrm{ac}}(H_{\omega})| \leq 4
\ee
holds for almost every $\omega$. However, we do not know if
this is true and trying to show this was a big motivation
to write this paper. To see that \eqref{eq:bigquestion}
cannot be a simple consequence of \eqref{eq:deisi} and \eqref{eq:kotani}
let $\mathcal{Z}$ be the complement of a Cantor set of
measure $> 4$ in $[-4,4]$. Then \eqref{eq:deisi} holds,
but $\ol{\mathcal{Z}}^{\mathrm{ess}} = [-4,4]$ is a set of
measure $8$.

Let us now discuss our main result and its relevance to this
conjecture. We begin by introducing some notation.
Let $(\Omega,\mu)$ be a probability space, $T:\Omega\to\Omega$ an
invertible ergodic transformation, and $f:\Omega\to \R$ a bounded
measurable function. For $\omega\in\Omega$ we introduce the
potentials $V_\omega(n)=f(T^n\omega)$. $H_{\omega} = \Delta + V_{\omega}$
denotes the Schr\"odinger operator with potential $V_{\omega}$.
Here $\Delta u (n) = u(n+1) + u(n-1)$ denotes the discrete Laplacian.
We introduce the transfer matrices
\be
 A_{N}(E, V_{\omega}) = \prod_{n=1}^{N} \begin{pmatrix} E - V_{\omega}(N- n) & -1
  \\ 1 & 0 \end{pmatrix}.
\ee
The Lyapunov exponent is then defined by
\be\label{eq:defLE}
 L(E) = \lim_{N\to\infty} \frac{1}{N} \int_{\Omega} \log\|A_{N}(E, V_{\omega})\| d\mu(\omega).
\ee
Our main result is

\begin{theorem}\label{thm:intA}
 There exists $(\Omega,T,f)$ such that $L(E)$ vanishes for generic $E \in [-4,4]$.
 Furthermore for every $\omega\in\Omega$, we have that
 $[-4,4] \subseteq \sigma(H_{\omega})$.
\end{theorem}

Our proof does not give any control on the measure of the
set $\mathcal{Z}$, where the Lyapunov exponent vanishes.
However, it is sufficient to show that

\begin{corollary}\label{cor:intA}
 We have that $L(E)$ is discontinuous on a set of
 $E$ of positive Lebesgue measure.
\end{corollary}

\begin{proof}
 Denote by $A$ the set of $E_0 \in [-4,4]$ such that
 $L(E)$ is continuous at $E_0$. By Theorem~\ref{thm:intA},
 we have that $A \subseteq \mathcal{Z}$. Hence
 by \eqref{eq:deisi}, we have that $|A| \leq 4$.
 In particular, the Lyapunov exponent is discontinuous
 on $B = [-4,4] \setminus A$ with measure $|B| \geq 4$.
\end{proof}

Our proof of Theorem~\ref{thm:intA} is done by constructing
a limit-periodic potential explicitly with these properties.
We will discuss the details in the next section, where we also
show that for a dense set of limit-periodic potentials the
Lyapunov exponent has at least one discontinuity.
This is achieved by adapting the argument of Johnson \cite{john}
to the discrete setting.

We furthermore wish to point out that Thouvenot has constructed
in \cite{thuv} discontinuity points of the Lyapunov exponent
for matrix-valued cocycles.

\bigskip

Let us now discuss the relevance to whether \eqref{eq:bigquestion}
holds. One can show that if $A$ is a closed set,
then $|\ol{A}^{\mathrm{ess}}| = |A|$. Now, the
existence of large sets of discontinuity of the
Lyapunov exponent stops us from concluding that
$\mathcal{Z}$ is closed and thus \eqref{eq:bigquestion}.

\bigskip

Artur Avila has informed us of an alternative proof of
Theorem~\ref{thm:intA}, which is based on the work of Bochi and
Viana \cite{bv}. The main idea is that if $T$ has plenty of periodic
points, then one can ensure that $|\sigma(\Delta + f(T^n \omega))| >
4$ for almost every $\omega$. Then one uses to show the results of
\cite{bv} to conclude that by perturbing $f$ slightly the Lyapunov
exponent $L(E)$ vanishes on a dense set in $\sigma(\Delta + f(T^n
\omega))$, and still $|\sigma(\Delta + f(T^n \omega))| > 4$. Then
discontinuity of the Lyapunov exponent follows as in the proof of
Corollary~\ref{cor:intA}.

%%%%%%%%%%%%%%%%%%%%%%%%%%%%%%%%%%%%%%%%%%%%%%%%%%%%%%%%%%%%%%%%%%%%%%%%%%%%%%%%%%%%
%
%
%
%
%

\section{Limit-periodic potential and further results}

We will begin this section by discussing some basics
about limit-periodic potentials. For further informations,
we refer to the appendix of the paper \cite{as} of
Avron and Simon, and to the survey by Gan \cite{g}.

Given a bounded sequence $V: \Z \to \R$, we denote by
$\Omega_V$ the hull of its translates. That is
\be
 \Omega_V = \ol{\{V_m,\quad m \in \Z\}}^{\ell^{\infty}(\Z)},
\ee
where $V_m(n) = V(n - m)$. If $\Omega_V$ is compact
in the $\ell^{\infty}(\Z)$ topology, then $V$ is
called {\em almost-periodic}.
The shift map on $\ell^\infty(\Z)$ becomes a translation
on the group $\Omega_V$ and it is uniquely
ergodic with respect to the Haar measure of $\Omega_V$.

$V$ is called {\em limit-periodic}, if there exists a sequence
of periodic potentials $V^k$ such that
\be
 V = \lim_{k \to \infty} V^k
\ee 
in the $\ell^\infty(\Z)$ topology. It should be remarked
that limit-periodic $V$ are almost-periodic. In fact, then
$\Omega_V$ has the extra structure of being a Cantor group,
see \cite{g} for details.

The proof of Theorem~\ref{thm:intA}
will proceed by explicitly constructing such
sequences of $V^k$. We will furthermore denote
by $L(E, V^k)$ the Lyapunov exponent of such a
sequence of periodic potentials.

Let us recall some properties of a $p_k$ periodic potentials $V^k$.
\begin{enumerate}
 \item The spectrum $\sigma(\Delta + V^k)$ is a finite union of intervals.
  That is
  \be
   \sigma(\Delta + V^k) = \bigcup_{j=1}^{g} [\alpha_j, \beta_j],
  \ee
  where $g \geq 2$ and $\alpha_j < \beta_j$.
 \item The Lyapunov exponent $L(E, V^k)$ is continuous
  and vanishes on $L(E, V^k)$.
\end{enumerate}
It should furthermore be pointed out that, one always
has $|\sigma(\Delta + V^k)| \leq 4$.

Let us now comment further on Theorem~\ref{thm:intA}.

\begin{remark}
 The potential constructed in Theorem~\ref{thm:intA}
 is limit-periodic. In particular, we provide an example
 of a limit-periodic potential whose spectrum contains an interval.
\end{remark}

This is not the first known example with this property,
since P\"oschel provided some in \cite{poesch}. However,
our construction has the advantage of being relatively
elementary in comparison to P\"oschel's KAM-type proof.

\bigskip

In order to state our other result,
we denote by $\mathcal{L}$ the space of all limit-periodic
potentials. We have

\begin{theorem} \label{thm:main2}
 There is a dense set of $V$ in $\mathscr{L}$ for which the Lyapunov
 exponent $L(E)$ is discontinuous.
\end{theorem}

Here the dense set cannot be improved to a generic set, since
Damanik and Gan in \cite{dg} show that there is a generic set of $V$
in $\mathscr{L}$ such that the Lyapunov exponent is
continuous.

Furthermore, we should point out that the proof of the previous
theorem largly parallels the construction of Johnson in \cite{john}
of similar examples in the continuum setting. However, the observation
of denseness is new.

\bigskip

Define the individual Lyapunov exponent by
\be\label{eq:defolLE}
 \ol{L}(E,\omega) = \limsup_{N\to\infty} \frac{1}{N} \log\|A_{N}(E, V_{\omega})\| .
\ee
An important aspect of our construction will be to
replace the Lyapunov exponent $L(E)$ defined in \eqref{eq:defLE}
by $\ol{L}(E, \omega)$. This is possible by the following result.

\begin{proposition}
 Assume that $T: \Omega\to\Omega$ is uniquely ergodic.
 For each $\omega\in \Omega$, there exists a set $\mathcal{E}_\omega$
 of zero Lebesgue measure such that for
 $E \in \R\setminus\mathcal{E}_{\omega}$
 \be
  L(E) = \ol{L}(E,\omega).
 \ee
\end{proposition}

\begin{proof}
 This is a consequence of Theorem~2.1. in \cite{k}.
\end{proof}

Thus, we obtain that for fixed $\omega$
\be\label{eq:lyapequal}
 L(E) = \lim_{N \to \infty} \frac{1}{N} \log\|A_{N}(E, V_{\omega})\|
\ee
for almost every $E$. We will furthermore need Thouless' formula
\be\label{eq:thouless}
 L(E) = \int \log|t - E| dN(t),
\ee
where $N(t)$ is the integrated density of states.
See for example Theorem~5.15 in Teschl's book \cite{tjac}.

%%%%%%%%%%%%%%%%%%%%%%%%%%%%%%%%%%%%%%%%%%%%%%%%%%%%%%%%%%%%%%%%%%%%%%%%%%%%%%%%%%%%
%
%
%
%
%

\section{Proof of Theorem~\ref{thm:intA}}
In order to simplify the notation, we introduce

\begin{definition}\label{def:epsdense}
 A collection of open intervals $\Sigma$ is called $\eps$-dense in $[-4,4]$
 if for $E \in [-4,4]$, there exists $I \in \Sigma$ such that $I \subseteq [E - \eps, E + \eps]$.
\end{definition}

We will construct a sequence of periodic potentials $V^k$ with the
following properties:
\begin{enumerate}
 \item $V^k$ is $p_k$ periodic and $\|V^{k} - V^{k-1}\|_{\infty} \leq \frac{1}{2^{k-1}}$.
 \item $V^{k}(n) = V^{k-1}(n)$ for $0 \leq n \leq p_{k-1}$.
 \item For $1 \leq l \leq k$ and $E \in \sigma(\Delta + V^l)$
  \be
   \frac{1}{p_k} \log\|A_{p_k}(E, V^{k})\| \leq \sum_{s = l+1}^{k+1} \frac{1}{2^s}.
  \ee
 \item There is a set $\Sigma_k$ consisting of open intervals $I \subseteq \sigma(\Delta + V^k)$
   which is $2^{-k}$-dense in $[-4,4]$.
 \item For each $I_{k-1} \in \Sigma_{k-1}$ there exists $I_{k} \in \Sigma_k$
   such that $I_k \subseteq I_{k-1}$.
\end{enumerate}

We begin by showing that the existence of such $V^k$ implies
Theorem~\ref{thm:intA}. From (i), one obtains that the limit \be
 V(n) = \lim_{k\to\infty} V^{k}(n)
\ee
exists uniformly and is bounded.

\begin{lemma}
 Let $l \geq 1$, then for $E \in \sigma(\Delta + V^l)$
 \be
  \liminf_{N \to \infty} \frac{1}{N} \log\|A_N(E, V)\| \leq \frac{1}{2^l}.
 \ee
\end{lemma}

\begin{proof}
 By (ii), we have $\frac{1}{p_k} \log\|A_{p_k}(E, V)\| = \frac{1}{p_k} \log\|A_{p_k}(E, V^k)\|$.
 By (iii), we obtain that for $E \in \sigma(\Delta + V^l)$ and $k \ge l$
 $$
  \frac{1}{p_k} \log\|A_{p_k}(E, V)\| \leq \sum_{s = l+1}^{k+1} \frac{1}{2^s} \leq \frac{1}{2^l},
 $$
 which implies the claim.
\end{proof}

Combining this lemma with \eqref{eq:lyapequal}, we obtain that
for almost every $E \in \sigma(\Delta + V^l)$
\be\label{eq:LEleq2l}
 L(E)\leq \frac{1}{2^l}.
\ee
We also have

\begin{proposition}
For $E_0 \in [-4,4]$ we have
 \be\label{eq:liminfle}
  \liminf_{E \to E_0} L(E) = 0.
 \ee
\end{proposition}

\begin{proof}
 Pick any $E_0 \in [-4,4]$ and $k \geq 1$.
 By (iv), we can choose $I_k \in \Sigma_k$ such that $I_k \subseteq [E_0 - \frac{1}{2^k}, E_0 + \frac{1}{2^k}]$.
 By (v), we can choose a sequence of intervals
 $$
  I_k \supseteq I_{k+1} \supseteq \dots
 $$
 Let $\hat{E} \in \bigcap_{l\geq k} \ol{I_l}$. Since $\hat{E} \in \ol{I_k}$,
 we have  $|E_0 - \hat{E}| < \frac{1}{2^k}$.
 By \eqref{eq:LEleq2l}, we can choose $E_l \in I_l$ with $|E_l - \hat{E}| \leq \frac{1}{2^l}$
 and $L(E_l) \leq \frac{1}{2^l}$.
 Hence
 $$
  0 \leq \liminf_{E \to \hat{E}} L(E) \leq \lim_{l \to \infty} L(E_l) = 0,
 $$
 showing \eqref{eq:liminfle}. Since $k \geq 1$ was arbitary, the claim follows.
\end{proof}

We now come to

\begin{proof}[Proof of Theorem~\ref{thm:intA}]
 Let 
 $$
  L_N(E) = \frac{1}{2^N} \int_{\Omega} \log\| A_{2^N}(E,V_\omega) \| d\mu(\omega).
 $$
 $L_N(E)$ is continuous in $E$ and decreasing in $N$.
 In particular, this implies
 $$
  L(E) = \inf_{N \geq 1} L_N(E).
 $$
 Introduce 
 $$
  U_l = \bigcup_{N\geq 1} \left\{E:\quad L_N(E) < \frac{1}{l} \right\}.
 $$
 Since $L_N(E)$ is continuous, $U_l$ is open.
 Let us now show that $U_l$ is also dense. Let $E_0 \in [-4,4]$ and 
 $\eps > 0$. By the previous proposition, there exists 
 $E \in [E_0-\eps, E_0+\eps]$ such that $L(E) < \frac{1}{2 l}$.
 Furthermore, there must be an $N_0 \geq 1$ such that
 $|L(E) - L_{N_0}(E)| < \frac{1}{2 l}$ and thus that
 $L(E) < \frac{1}{l}$. This implies $E \in U_l$ and thus
 that $U_l$ is dense.

 Introduce 
 $$
  U = \bigcap_{l \geq 1} U_l.
 $$
 Since each of the $U_l$ is open and dense, we have by
 the Baire category theorem that also $U$ is dense.
 The claim follows.
\end{proof}

In the following subsections, we will explain how to construct
$V^k$ given $V^{k-1}$. In the next subsection, we will construct
a sequence of potentials $\hat{V}_m$  such that properties
(i), (ii), (iv), and (v) hold. Then, we will show in Subsection~\ref{sub:ensuringiii}
that (iii) holds if $m$ is chosen large enough.

\subsection{A sequence of potentials such that (i), (ii), (iv) and (v)
hold}
We will need the following lemma, describing the generalized eigenfunctions
of a periodic operator. A proof can be given using that $A_p(E, V)$ has
an eigenvalue of modulus $1$ for $E \in \sigma(\Delta + V)$ and we omit
it for brevity.

\begin{lemma} \label{lem:perisolu}
 Let $V$ be a $p$ periodic potential and $E \in \sigma(\Delta + V)$.
 There is a bounded solution $\psi$ of $(\Delta + V)\psi = E \psi$ such that
 for every $m \in \Z$
 \be \label{eq:perisolu}
 \sum^{p}_{k=1} |\psi(m + k)|^2 = 1.
 \ee
\end{lemma}

Let $I^{1}, \dots, I^{L}$ be an enumeration of the intervals in
$\Sigma_{k-1}$. For each $I^{l}$, we choose a subinterval
$\ti{I}^{l}$ of length $< \frac{1}{2^{k+1}}$. Denote by
$\widetilde{\Sigma}_{k-1}$ the collection of intervals $\ti{I}^l$.
Introduce $\delta$ by \be
 \delta = \min_{I \in \widetilde{\Sigma}} |I|.
\ee
Clearly, $0 < \delta < \frac{1}{2^{k+1}}$.

Choose $m_0$ so large that $\delta \sqrt{ m_0} > 4$ and let
$\ti{p}_{k-1} = m_0 p_{k-1}$. We will treat $V^{k-1}$ as a
$\ti{p}_{k-1}$ periodic potential. Write
\be
 \Lambda_{j,m} = [m \ti{p}_{k-1} + (j + 4) \ti{p}_{k-1}, m \ti{p}_{k-1} +(j + 5) \ti{p}_{k-1} - 1].
\ee
Introduce the potentials $\hat{V}_m$ by
\be
 \hat{V}_m(n) = \begin{cases} V^{k-1}(n), & 0 \leq n \leq m \ti{p}_{k-1} - 1 ; \\
  V^{k-1}(n) + \frac{j}{2^{k+1}}, & n \in \Lambda_{j,m}\quad -4\leq j \leq 3.
 \end{cases}
\ee
Note that $\hat{V}_m$ will be $\hat{p}_m = m \ti{p}_{k-1} + 8 \ti{p}_{k-1}$ periodic.
We will later let $V^{k} = \hat{V}_m$ for some large $m$.

We see that the claimed properties (i) and (ii) are straightforward.
It remains to prove (iv) and (v).

\begin{lemma} \label{lem:subinterval}
 For each $I \in \widetilde{\Sigma}_{k-1}$ and $-4\leq j\leq 3$,
 there exists an open interval $J$ such that
 \be\label{eq:propJ}
  J \subseteq I + \frac{j}{2^{k+1}} \text{ and }
  J \subseteq \sigma(\Delta + \hat{V}_m).
 \ee
\end{lemma}

\begin{proof}
 Let $I = (E_-,E_+)$ and $\hat{E} = \frac{E_- + E_+}{2}$.
 By Lemma~\ref{lem:perisolu}, there exists a function $\psi$
 such that
 $$
  (\Delta + V^{k-1}) \psi = \hat{E} \psi,\quad \text{and}\quad
  \sum^{p_{k-1}}_{n = 1} |\psi(n)|^2 = 1.
 $$
 Let $\varphi$ be the restriction of $\psi$ to $\Lambda_{j,m}$.
 A computation shows $\|\varphi\| = \sqrt{m_0}$ and
 $$
  (\Delta + \hat{V}_{m} - \hat{E}  - \frac{j}{2^{k+1}}) \varphi(n) =
  \begin{cases} \psi(m \ti{p}_{k-1} + (j + 4) \ti{p}_{k-1}), & n = m \ti{p}_{k-1} + (j + 4) \ti{p}_{k-1} - 1;\\
  - \psi(m \ti{p}_{k-1} + (j + 4) \ti{p}_{k-1} - 1), & n = m \ti{p}_{k-1} + (j + 4) \ti{p}_{k-1}; \\
  - \psi(m \ti{p}_{k-1} + (j + 5) \ti{p}_{k-1} ), & n = m \ti{p}_{k-1} + (j + 5) \ti{p}_{k-1} -1;\\
  \psi(m \ti{p}_{k-1} + (j + 5) \ti{p}_{k-1} - 1), & n = m \ti{p}_{k-1} + (j + 5) \ti{p}_{k-1};\\
  0, & \text{otherwise}.  \end{cases}
 $$
 So we have
 \be
  \frac{\|(\Delta + \hat{V}_m - \hat{E} - \frac{j}{2^{k+1}}) \varphi \|}{\| \varphi \|} \leq \frac{2}{\sqrt{m_0}} < \frac{\delta}{2}.
 \ee
 The above inequality implies that
 $$
 \dist(\hat{E} + \frac{j}{2^{k+1}},\sigma(\Delta + \hat{V}_m)) < \frac{\delta}{2}.
 $$
 Hence, we see that $\sigma(\Delta  + \hat{V}_m) \cap (I + \frac{j}{2^{k+1}})$
 is non empty since $|I + \frac{j}{2^{k+1}}| \geq \delta$.
 Since $\sigma(\Delta + \hat{V}_m)$ consists of bands,
 we may choose an open interval $J$ such that \eqref{eq:propJ} holds.
\end{proof}

We denote by $\Sigma_k$ the collection of open
intervals $J$ obtained from the previous lemma for all possible choices
of $I$ and $j$. It is clear that property (v) holds.

\begin{proof}[Proof of Property (iv)]
 Given any $E \in [-4,4]$, we may find an $I
 \in \widetilde{\Sigma}_{k-1}$ such that
 $$
  I \subseteq [E - \frac{1}{2^{k-1}}, E + \frac{1}{2^{k-1}}].
 $$
 Let $I = (E_-,E_+)$ and $\hat{E} =\frac{E_- + E_+}{2}$.
 Since $\hat{E} \in I$ and
 $$
  [E - \frac{1}{2^{k-1}}, E + \frac{1}{2^{k-1}}] = \bigcup_{j=-4}^{3} \left[E + \frac{j}{2^{k+1}}, E + \frac{j+1}{2^{k+1}}\right],
 $$
 we can find $-4 \leq j \leq 3$ such that
 $$
  \hat{E} \in \left[E - \frac{j}{2^{k+1}}, E - \frac{j+1}{2^{k+1}}\right]
 $$
 By the construction of $\Sigma_{k}$ in Lemma~\ref{lem:subinterval},
 there exists $J \in \Sigma_{k}$  such that $J \subseteq I + \frac{j}{2^{k+1}}$.
 By $|I| \leq \frac{1}{2^{k+1}}$, we obtain
 $$
  J \subseteq I + \frac{j}{2^{k+1}} \subseteq
  \left[\hat{E} + \frac{j-1}{2^{k+1}}, \hat{E} + \frac{j+1}{2^{k+1}}\right]
  \subseteq \left[E - \frac{1}{2^{k}}, E + \frac{1}{2^k}\right],
 $$
 which finishes the proof.
\end{proof}

\subsection{Ensuring (iii)}\label{sub:ensuringiii}

\begin{lemma}\label{lem:transmatsmall}
 Let $V$ be a $p$ periodic potential.
 For any $\mu > 0$, there exists $M \geq 1$ such that
 for $E \in \sigma(\Delta + V)$ and $N \geq M$, we have
 \be
  \frac{1}{N} \log\|A_{N}(E, V)\| \leq \mu.
 \ee
\end{lemma}

\begin{proof}
 Introduce the continuous functions
 $$
  f_i(E) = \frac{1}{2^i} \log \| A_{2^i} (E,V)\|.
 $$
 One can easily check that $f_{i+1}(E) \leq f_i(E)$ and
 that for $E \in \sigma(\Delta + V)$, we have $\lim_{i \to \infty} f_i(E) = L(E,V) = 0$.
 Hence, we obtain by Dini's therorem that there exists $M_1 \geq 1$ such that for $i \geq M_1$
 and $E \in \sigma(\Delta +V)$
 $$
  \frac{1}{2^i} \log \| A_{2^i} (E,V)\| \leq \frac{\mu}{2}.
 $$
 Let $N = q 2^{M_1} + r$, where $q \geq 1$ and $0 \leq r \leq 2^{M_1} -1$.
 For $E \in \sigma(\Delta +V)$, we get
 \begin{align*}
  \frac{1}{N} \log \|A_{N} (E,V)\| & \leq \frac{1}{q 2^{M_1} + r} \left ( \log \| A_{q 2^{M_1}} (E,V)\| + \log\| A_{r} (E,V)\|\right) \\
 & \leq \frac{\mu}{2}  + \frac{\log\| A_{r} (E,V)\|}{q 2^{M_1} + r} .
 \end{align*}
 Choose  $M_2 \geq 1$ so large that
 for $E \in \sigma(\Delta +V)$ and $0 \leq r \leq 2^{M_1} - 1$
 $$
 \frac{\log\p A_{r} (E,V)\p }{M_2 2^{M_1} + r} \leq \frac{\mu}{2}.
 $$
 The claim follows by taking $q \geq M_2$ or equivalently $N \geq M = M_2 \cdot 2^{M_1}$.
\end{proof}

By this lemma, we can ensure (iii) for $l = k$ as long as $m$ large enough.
Let $1 \leq l \leq k - 1$. By assumption,
$$
 \frac{1}{p_{k-1}} \log\|A_{p_{k-1}}(E, V^{k-1})\| \leq \sum_{s = l+1}^{k} \frac{1}{2^s}.
$$
By submultiplicativity, we get
$$
 \frac{1}{\ti{p}_{k-1}} \log\|A_{\ti{p}_{k-1}}(E, V^{k-1})\| \leq \sum_{s = l+1}^{k} \frac{1}{2^s}.
$$
We recall that $\hat{V}_m$ is a $\hat{p}_m = (m + 8)\ti{p}_{k-1}$ periodic potential,
such that $\hat{V}_m(n) = V^{k-1}(n)$ for $0 \leq m \ti{p}_{k-1}-1$.
Thus, we have
\be
 A_{\hat{p}_m}(E, \hat{V}_m) = (A_{\ti{p}_{k-1}}(E, V^{k-1}))^m \cdot \ti{A},
\ee
where $\ti{A}$ is some fixed matrix, whose norm is independent of $m$.
Hence, we obtain
$$
 \frac{1}{\hat{p}_m} \log\|A_{\hat{p}_m}(E, \hat{V}_m)\| \leq \sum_{s = l+1}^{k} \frac{1}{2^s} + \frac{1}{\hat{p}_m} \log\|\ti{A}\|.
$$
The claim now follows by choosing $m$ large enough.

\subsection{Construction of the initial potential.}

Last, we construct an initial potential $V^0$.
Define for $L \geq 1$, the potential $\ti{V}_L$ by
\be
 \ti{V}_L (n) = \begin{cases} 2, & 0 \leq n \leq L-1 ; \\
  -2, & L \leq n \leq 2L-1
 \end{cases}
\ee and continue it to be $2L$ periodic. For $k \in [0,\pi]$ the
function $\varphi(n) = \E^{i kn}$ solves $\Delta \varphi = 2 \cos(k)
\varphi$. Set $E = 2 \cos(k)$ and \be
 \psi(n) = \begin{cases} \varphi(n), & 0 \leq n \leq L-1 ; \\
  0, & \text{otherwise}.
 \end{cases}
\ee
Clearly, $\| \psi \| = \sqrt{L}$, and
$$
\| (\Delta + V^0 - E-2) \psi \| \leq 2.
$$
This implies that $\mathrm{dist}(E+2,\sigma(\Delta+V^0)) \leq \frac{2}{\sqrt{L}}.$
When $L > 4$, we have
$$
 \mathrm{dist}(E+2,\sigma(\Delta+V^0)) < 1.
$$
So for any $E_0 \in [0,4]$ we can find an interval $I \subset \sigma(\Delta + V^0)$
such that $I \subseteq [E_0 - 1, E_0 + 1].$

Similarly, we conclude that
for any $E_0 \in [-4,0]$ there is an interval $I \subset \sigma(\Delta + V^0)$
such that $I \subseteq [E_0 - 1, E_0 + 1].$ Thus, for $V^0$ we can find a collection $\Sigma_0$
of intervals $I \subset \sigma(\Delta + V^0)$ which is $1$-dense in $[-4,4]$. (iv) follows.

By Lemma~\ref{lem:transmatsmall}, (iii) follows since we can treat $V_0$ as a $p_0$
periodic where $p_0 = 2 L m$, and $m \in \Z^+$ is large.
There is nothing to check for (i), (ii), and (v).

%%%%%%%%%%%%%%%%%%%%%%%%%%%%%%%%%%%%%%%%%%%%%%%%%%%%%%%%%%%%%%%%%%%%%%%%%%%%%%%%%%%%%%%%%%%%%%%%%%%%%%%%%%%%%%%%%%%
%
%
%
%
\section{Proof of Theorem~\ref{thm:main2}}

The first step in the proof of Theorem~\ref{thm:main2}
is to prove the following proposition, which we postpone
to a later subsection.

\begin{proposition}\label{prop:perturb}
 Let $V_0$ be a $p_0$ periodic potential such that
 \be
  E_0 = \inf(\sigma(\Delta + V_0)) > 0.
 \ee
 Introduce $\gamma = \frac{1}{2} L(0, V^0)$. There exists
 a limit-periodic potential $V$ such that $\|V - V_0\| \leq 5 E_0$,
 $\inf(\sigma(\Delta + V)) = 0$, and
 \be
  \limsup_{E \to 0} L(E) \geq \gamma > \liminf_{E \to 0} L(E) = 0.
 \ee
 In particular, the Lyapunov exponent $L(E, V)$ is discontinuous
 at $0$.
\end{proposition}

We are now ready for

\begin{proof}[Proof of Theorem~\ref{thm:main2}]
 First note that the periodic potentials are dense in
 the space $\mathscr{L}$ of all limit periodic potentials.

 Let $V_0$ be any periodic potential and $\eps > 0$.
 Introduce the potential
 $$
  \ti{V}_0 = V_0 - \inf(\Delta + V_0) + \frac{\eps}{5}.
 $$
 One can check that $\ti{V}_0$ satisfies the assumptions
 of the previous proposition, and we thus obtain a potential
 $\ti{V}$ such that the Lyapunov exponent of $\ti{V}$
 is discontinuous and $\|\ti{V}_0 - \ti{V}\| \leq \eps$.

 We now see that $V = \ti{V} + \inf(\Delta + V_0) - \frac{\eps}{5}$
 satisfies that its Lyapunov exponent is discontinuous and
 $\|V - V_0\| \leq \eps$. Hence, the potentials with
 discontinuous Lyapunov exponent are dense.
\end{proof}

\subsection{Proof of Proposition~\ref{prop:perturb}}

In the following subsection, we will construct a sequence of
potentials $V^k$ with the following properties.

\begin{enumerate}
  \item $V^k$ is $p_k$ periodic and satisfies
   \be
    \|V^{k} - V^{k-1}\| \leq \frac{2}{5} E_{k-1},
   \ee
   and $V^k(n) = V^{k-1}(n)$ for $0 \leq n \leq p_{k - 1}$.
  \item The bottom of the spectrum $E_k = \inf \sigma(\Delta + V^k)$ satisfies
   \be
    \left(\frac{3}{5}\right)^k E_0 \leq E_k \leq \left(\frac{4}{5}\right)^k E_0.
   \ee
  \item The Lyapunov exponent at energy $0$ satisfies
   \be
    L(0, V^k) \geq \left(2 - \sum_{s = 1}^{k} \frac{1}{2^s} \right) \gamma.
   \ee
  \item For $1 \leq l \leq k$ and  every $E \in \sigma(\Delta + V^l)$
   \be
    \frac{1}{p_k} \log\|A_{p_k}(E, V^k)\| \leq  \sum_{s = l+1}^{k+1} \frac{1}{2^s}.
   \ee
 \end{enumerate}

From properties (i) and (ii), we see that
\be
 \|V^{k} - V^{k-1}\| \leq \left(\frac{4}{5}\right)^k E_0.
\ee
Hence the limit
\be
 V = \lim_{k \to \infty} V^k,
\ee
exists in $\ell^\infty(\Z)$ and $\|V - V^0\| \leq 5 E_0$.

\begin{lemma}
 We have that
 \be
  L(0) \geq \gamma.
 \ee
\end{lemma}

\begin{proof}
 By (ii), we have that $\inf(\sigma(\Delta + V^k)) > 0$.
 Hence by Thouless' formula \eqref{eq:thouless}
 $$
  L(E, V^k) = \int \log|t - E| dN^k(t),
 $$
 we have that $E \mapsto L(E, V^k)$ is decreasing in $E \leq 0$.

 This and property (iii) imply for $E < 0$ and $k \geq 1$ that
 $$
  L(E, V^k) \geq \gamma.
 $$
 Since $N^k \to N$ weakly and $\log|t - E|$ is a bounded
 and continuous function for $E < 0$, we also obtain
 $$
  L(E) \geq \gamma
 $$
 for $E < 0$. Thouless formula even implies that
 $$
  L(0) = \lim_{E\uparrow 0} L(E).
 $$
 This implies the claim.
\end{proof}

Similarly as for \eqref{eq:LEleq2l} in the previous section,
we have that properties (i) and (iv) imply that
\be
 L(E) \leq \frac{1}{2^{k}}
\ee
for almost every $E \in \sigma(\Delta + V^k)$.
Hence, we have that
\be
 \limsup_{E \to 0} L(E) > \liminf_{E \to 0} L(E) = 0.
\ee
This finishes the proof of Proposition~\ref{prop:perturb}.

\subsection{Construction of the sequence of potentials}

In order to construct the sequence of potentials $V^k$,
we will prove the following lemmas. These imply
the existence of $V^k$ given $V^{k-1}$ by applying
them to $V = V^{k-1}$.

\begin{lemma} \label{lem:lowerspec}
 Let $V$ be a $p$ periodic potential and $E_0 = \inf\sigma(\Delta + V).$
 Let $\tilde{p} = m_0 p$ for sufficiently large $m_0$.
 Define for $1 \leq l \leq h$ the intervals $I_l$ by
 \be
  I_l = [m \tilde{p} + (l-1) \tilde{p}, m \tilde{p} + l \tilde{p} - 1].
 \ee
 Define the $\hat{p} = (m + h ) \tilde{p}$ periodic potential $\hat{V}_{m,h}$  by
 \be\label{eq:defhatV}
  \hat{V}_{m,h} (n) = \begin{cases}
   V(n) & [0, m \tilde{p} -1];\\
   V(n) -  \frac{2 E_0}{5} & n \in I_l,\quad 1 \leq l \leq h. \end{cases}
 \ee
 Then we have
 \be
  \frac{3E_0}{5} \leq \inf\sigma(\Delta + \hat{V}_{m,h}) \leq \frac{4 E_0}{5}.
 \ee
\end{lemma}

\begin{proof}
As in the proof of Lemma~\ref{lem:subinterval}, let $E_{-} = E_0$
and $\delta = \frac{E_0}{5}$. Pick $\hat{E}$ in the first band of
$\sigma(\Delta+V)$ such that $\hat{E} - E_{0} < E_{0}/10$. Also,
$m_0$ will be picked sufficiently large so that $\delta \sqrt{m_0} >
4$. Then, we will get
\be
\dist(\hat{E}- \frac{2E_0}{5}, \sigma(\Delta + \hat{V}_{m,h})) \leq \frac{\delta}{2} = \frac{E_0}{10}.
\ee
The lemma follows.
\end{proof}

We can view the $\hat{p}$ period of $\hat{V}_{m,h}$ as a
concatenation of two parts:
\begin{itemize}
 \item $m$ pieces of the $\ti{p}$ period of $V$.
 \item $h$ pieces of the $\ti{p}$ period of $\tilde{V}$, where $\tilde{V} = V - \frac{2E_0}{5}.$
\end{itemize}
Denote $\ti{A} = A_{\ti{p}}(E,\ti{V})$,
then we have
\be
 A_{\hat{p}}(E,\hat{V}_{m,h}) =  (\ti{A})^h \cdot (A_{\ti{p}}(E, V))^{m}.
\ee
Let $E \notin \inf(\sigma(\Delta + V))$.
Then we can chose two normalized vectors $v,u \in C^2$, such
that
\be
 A_{\tilde{p}}(E, V) v = \E^{\tilde{p} L(E,V)} v,\qquad
 A_{\ti{p}}(E, V) u = \E^{-\ti{p} L(E,V)} u.
\ee
Denote by $v^{\perp}= a v + b u$ a vector orthonormal to $v$.
We will first show

\begin{lemma}
 There exists $h \in \{1,2\}$ such that
 \be
  \spr{v}{\ti{A}^h v} + a \spr{v^{\perp}}{\ti{A}^h v} \neq 0.
 \ee
\end{lemma}

\begin{proof}
 From the Cayley--Hamilton theorem, we have that
 $\ti{A}^2 - \tr(\ti{A}) \ti{A} + I = 0$.
 Taking $\spr{v}{. v}$ and $\spr{v^{\perp}}{ . v}$,
 we obtain
 \begin{align*}
  \spr{v}{\ti{A}^2v} - \tr(\ti{A}) \spr{v}{ \ti{A} v} + 1 &= 0, \\
  \spr{v^{\perp}}{\ti{A}^2v} - \tr(\ti{A}) \spr{v^{\perp}}{ \ti{A} v} &= 0.
 \end{align*}
 Multiplying the second equation by $a$ and adding the two together,
 we obtain
 $$
  \Big(\spr{v}{\ti{A}^2v} + a \spr{v^{\perp}}{\ti{A}^2v}\Big)
  - \tr(\ti{A}) \Big(\spr{v}{ \ti{A} v} + a \spr{v^{\perp}}{ \ti{A} v}\Big) = -1.
 $$
 Hence, the claim follows.
\end{proof}

We now come to

\begin{lemma} \label{lem:lowle}
 Let $E \notin \sigma(\Delta + V)$.
 Then there exists $h \in \{1,2\}$ such that
 \be
  L(E, \hat{V}_{m,h}) \to L(E, V)
 \ee
 as $m \to \infty$.
\end{lemma}

\begin{proof}
 Let $h$ be as in the previous lemma.
 The lower bound can be obtained by a similar argument as
 in Subsection~\ref{sub:ensuringiii}.

 By (7.10) in \cite{tjac} and a computation, we have
 \begin{align*}
  L(E, \hat{V}_{m,h}) & = \frac{1}{\hat{p}} \log \left (\left |\frac{\tr(A_{\hat{p}}(E,\hat{V}_{m,h}))}{2} \right | +
  \sqrt {\left (\frac{\tr(A_{\hat{p}}(E,\hat{V}_{m,h}))}{2}\right )^2 - 1} \right )\\
  & \geq \frac{1}{\hat{p}} \log{\left | \tr(\ti{A}^h (A_{\ti{p}}(E,V))^m ) \right |} - \frac{\log(2)}{\hat{p}}.
 \end{align*}
 Let us now evaluate $\tr(\ti{A}^h (A_{\ti{p}}(E,V))^m )$. We have
 \begin{align*}
  \tr(\ti{A}^h (A_{\ti{p}}(E,V))^m )
  &=\spr{v}{\ti{A}^h (A_{\ti{p}}(E,V))^m v} + \spr{v^{\perp}}{\ti{A}^h (A_{\ti{p}}(E,V))^m v^{\perp}} \\
  &= \E^{m\ti{p} L(E, V)}\spr{v}{\ti{A}^h v} + a \E^{m\ti{p} L(E, V)} \spr{v^{\perp}}{\ti{A}^h v}\\
  &\qquad +b \E^{- m\ti{p} L(E, V)}\spr{v^{\perp}}{\ti{A}^h u} \\
  &= \E^{m\ti{p} L(E, V)} \left(\spr{v}{\ti{A}^h v} + a \spr{v^{\perp}}{\ti{A}^h v}\right) + o(1).
 \end{align*}
 Combining this with the previous formula, we obtain that
 $$
  L(E, \hat{V}_{m,h}) \geq L(E,V) + o(1).
 $$
 Hence, we see that the claim holds, if we choose $m$ large enough.
\end{proof}

In order to show the existence of $V^k$ such that (i) to (iv) hold.
Use Lemma~\ref{lem:lowerspec} to find a sequence of potentials
$\hat{V}_{m,h}$ such that properties (i) and (ii) hold.
The previous lemma implies the existence of $h \in \{1,2\}$ such
that property (iii) for $m \geq 1$ large enough. Finally,
by arguments similar to the ones in Subsection~\ref{sub:ensuringiii}, we can show
that property (iv) holds for $m \geq 1$ large enough.
This finishes the proof of the existence of the
sequence $V^k$ and so also of Proposition~\ref{prop:perturb}.

\section*{Acknowledgments}

We thank Artur Avila and David Damanik for useful discussions and Svetlana Jitomirskaya for
informing us of \cite{thuv}.

\end{document}